\documentclass[journal]{IEEEtran}
\usepackage{graphics} 
\usepackage{epsfig} 
\usepackage{amsmath} 
\usepackage{amsthm}
\usepackage{amssymb}  

\usepackage{mathrsfs}

\usepackage{subfigure}
\usepackage{color} 
\usepackage{cite}
\usepackage{mathtools}
\usepackage{mathrsfs}
\usepackage{txfonts}
\usepackage{amssymb}
\usepackage{multirow}
\usepackage{subfigure}

\newtheorem{assumption}{Assumption}[section]

\newcommand{\HEI}[1]{\bf Hybrid-EI}

\newtheorem{theorem}{Theorem}
\newtheorem{remark}{Remark}
\newtheorem{definition}{Definition}

\ifCLASSINFOpdf

\else

\fi

\begin{document}
%
\title{\LARGE \textbf{Event-Triggered Impulsive Control for Nonlinear Systems\\ with Actuation Delays}}
%
%
%

\author{Kexue~Zhang~~~~
        ~Elena~Braverman
\thanks{This work was supported by the Natural Sciences and Engineering Research Council of Canada (NSERC), the grant RGPIN-2020-03934. The first author was partially supported by a fellowship from the Pacific Institute of the Mathematical Sciences (PIMS), Canada.}
\thanks{K. Zhang is with the Department
of Mathematics and Statistics, Queen's University, Kingston, Ontario K7L 3N6, Canada (e-mail: kexue.zhang@queensu.ca).}
\thanks{E. Braverman is with the Department
of Mathematics and Statistics, University of Calgary, Calgary, Alberta T2N 1N4, Canada (e-mail: maelena@ucalgary.ca).}
 }

\maketitle
%
\begin{abstract}
This paper studies impulsive stabilization of nonlinear systems. We propose two types of event-triggering algorithms to update the impulsive control signals with actuation delays. The first algorithm is based on continuous event detection, while the second type makes decision about updating the impulsive control inputs according to periodic event detection. Sufficient conditions are derived to ensure asymptotic stability of the impulsive control systems with the designed event-triggering algorithms. Lower bounds of the time period between two consecutive events are also obtained, so that the closed-loop impulsive systems are free of Zeno behavior. That is to say that the pulse phenomena are excluded from the event-triggered impulsive control systems, in the community of impulsive differential equations. An illustrative example demonstrates effectiveness of the proposed algorithms and our theoretical results.

\end{abstract}
%

\section{Introduction}\label{Sec1}

\IEEEPARstart{I}{mpulsive} control is a control paradigm that uses impulses that are state abrupt changes over negligible time periods to control dynamical systems. It has been proved to be powerful in various control problems, such as consensus of multi-agent systems, secure communications, and pulse vaccination strategy in epidemic diseases (see, e.g., \cite{TY:2001,XL-KZ:2019} and the references therein). Most of the existing results on impulsive control problems focus on time-triggered control strategies. More specifically, the moments when the impulses happen, normally called \textit{impulse times}, are pre-scheduled which makes time-triggered control strategies simple to implement. Nevertheless, unnecessary impulsive control tasks could be executed which clearly is a waste of control efforts and communication resources. 

To improve the impulsive control efficiency, event-triggered impulsive control has been successfully developed recently, the idea of which is to determine the impulse times or the instants of updating the control signals by a certain event that occurs only when the system dynamics violates a well-designed triggering condition. Compared with time-triggered impulsive control, the impulse times are implicitly determined by the triggering condition, and the impulsive controller is only activated when the event-triggering condition is satisfied. The main difference between event-triggered impulsive control and conventional event-triggered control is as follows. The event-triggered controllers work in a sample-and-hold fashion, that is, the control signal is updated at each event time and remains unchanged until the next event is triggered (see, e.g., \cite{PT:2007,WH-KHJ-PT:2012}). The event-triggered impulsive control signals are instantaneous control inputs over negligible periods, and the control system is uncontrolled between two consecutive impulse times. Thus, event-triggered impulsive control is capable of significantly reducing the time on the execution of control tasks.

The past few years have witnessed a surge of interest in the study of impulsive control systems with event-triggered impulse times. In \cite{BL-DJH-ZS:2018}, input-to-state impulsive stabilization of nonlinear systems was studied with the impulse times determined by an event-triggering algorithm with three levels of events which makes the event-triggered impulsive controllers complicated to implement. The prescribed upper bound of the \textit{inter-event times}, which are the periods lying between two consecutive event times, potentially requires the impulse signals to occur more frequently, which could possibly trigger some unnecessary control tasks. The same control problem was also investigated in \cite{MG-ZA-LP:2019}, and the designed event-triggering scheme requires the inter-event times belong to a predetermined set of periods which dramatically limits the selection of the event times. Uniform stability and global asymptotic stability of impulsive systems were studied in \cite{XL-DP-JC:2020} with event-triggered impulse times.  
An extra forced sequence of impulse times is required to ensure asymptotic stability of the control systems which makes the proposed event-triggering algorithm complicated to implement. Such event-triggering method was then applied to the synchronization problem of dynamical networks in \cite{DP-XL:2020}. 

Time delays are ubiquitous in nature and exist widely in many practical systems (see \cite{EF:2014,JKH:1977}). Event-triggered impulsive control for time-delay systems has attracted an increasing interest in the community of control theory and engineering. To our best knowledge, exponential stabilization of general nonlinear systems with time delay was initially studied in \cite{XL-XY-JC:2020} by the method of event-triggered impulsive control, and then the proposed event-triggering algorithm was extended to investigate input-to-state stability of nonlinear systems in~\cite{XL-PL:2021}. However, no time-delay effects are considered with the impulsive controllers in these studies. The event-triggering algorithm in \cite{XL-XY-JC:2020} was also generalized to synchronize complex dynamical networks with coupling delays in \cite{XL-JC-XL-MAA-UAAJ:2020}, and impulsive controllers were designed with sensing delays which are communication delays in the controller-sensor pair. Besides the above-mentioned results, there are many interesting event-triggering algorithms designed for some particular impulsive control problems (see, e.g., \cite{XT-JC-XL:2019,MSA-RV-OMK:2019,WD-SYSL-TY-AVV:2017}). It should be noted that all event times coincide with the impulse times in the existing literature, that is, no communication delays are considered in the controller-actuator pair. Such delays are known as \textit{actuation delays}. Moreover, the event detection is based on a continuous sampling of the system states in~\cite{XL-XY-JC:2020,XL-PL:2021}. Motivated by the previous discussion, we consider the event-triggered impulsive control problem of nonlinear systems with actuation delays and construct event-triggering schemes based on continuous and periodic event detection, respectively.


In this paper, we focus on the impulsive control problem of nonlinear autonomous systems with actuation delays. Due to the existence of actuation delays, the moments when the impulsive control signals are updated are different from the impulse times. We design two types of event-triggering algorithms to determine the control updating times which are also called \textit{event times}. Our event-triggering algorithms have the following merits: (a) actuation delays are considered with the event-triggered impulsive controllers; (b) some of our event-triggering algorithms are based on periodic state sampling; 
(c) compared with the existing results (see, e.g., \cite{BL-DJH-ZS:2018,MG-ZA-LP:2019,XL-DP-JC:2020,DP-XL:2020,XL-XY-JC:2020,XL-JC-XL-MAA-UAAJ:2020}), our event-triggering algorithms do not need the memory of the system information at the previous event time to determine the coming event which makes the proposed algorithms easy to implement. 


\textit{Notation}. Let $\mathbb{R}$ denote the set of real numbers, $\mathbb{R}^+$ the set of non-negative real numbers, $\mathbb{N}$ the set of positive integers. Let $\mathbb{R}^{n}$ and $\mathbb{R}^{n\times n}$ denote $n-$dimensional and $n\times n-$ dimensional real spaces equipped with the Euclidean norm and the induced matrix norm, respectively, both represented by $\|\cdot\|$. Denote by $\mathcal{V}$ the set of locally Lipschitz continuous functions mapping $\mathbb{R}^n$ to $\mathbb{R}^+$. A continuous function $\alpha:\mathbb{R}^+\rightarrow \mathbb{R}^+$ is said to be of class $\mathcal{K}_{\infty}$ if $\alpha$ is strictly increasing, equals zero at zero, and satisfies $\alpha(s)\rightarrow \infty$ as $s\rightarrow \infty$. We denote by $\alpha^{-1}$ the inverse function of function $\alpha\in\mathcal{K}_{\infty}$.

\section{Preliminaries}\label{Sec2}

Consider the following impulsive control system:
\begin{eqnarray}\label{sys}
\left\{\begin{array}{ll}
\dot{x}(t)=f(x(t)), ~~t\not=t_k+\tau,\cr
\Delta x(t_k+\tau)=g(x(t_k)),~~k\in\mathbb{N},\cr
x(0)=x_0,
\end{array}\right.
\end{eqnarray}
where $x(t)\in\mathbb{R}^n$ is the system state; $x_0\in\mathbb{R}^n$ represents the initial state; $f$ is a continuous function mapping $\mathbb{R}^n$ to $\mathbb{R}^n$ and satisfying $f(0)=0$; function $g: \mathbb{R}^n \mapsto \mathbb{R}^n$ satisfies $g(0)=0$.  Hence, system~\eqref{sys} admits the trivial solution. The time sequence $\{t_k\}_{k\in\mathbb{N}}$ is to be determined based on the occurrence of certain events to be defined later. The impulsive control input $g(x)$ is updated at each \textit{event time} $t_k$. However, due to the existence of the actuation delay {$\tau\in\mathbb{R}^+$}, the impulsive control task is executed at time $t_k+\tau$ instead of the event time $t_k$. Thus, $\{t_k+\tau\}_{k\in\mathbb{N}}$ is the \textit{impulse time} sequence. In system~\eqref{sys}, $\Delta x(t)$ indicates the impulse or the state jump at time $t$ and is defined as $\Delta x(t): =x(t^+)-x(t^-)$, where $x(t^+)$ and $x(t^-)$ represent the right and left limits of $x$ at time $t$, respectively. Here, we assume $x$ is right continuous at each impulse time, that is, $x((t_k+\tau)^+)=x(t_k+\tau)$ for $k\in\mathbb{N}$.

The objective of this research is {to design appropriate event-triggering conditions to determine the event times and establish sufficient conditions on the impulsive control input $g$ and the actuation delay $\tau$, to ensure some stability properties of system~\eqref{sys}} which are formally defined as follows.

\begin{definition}\label{asymptotic-stability}
The trivial solution of system~\eqref{sys} is said to be 
\begin{itemize}
\item \textbf{stable}, if, for any $\varepsilon>0$, there exists a $\sigma:=\sigma(\varepsilon)>0$ such that $\|x_0\|<\sigma \Rightarrow \|x(t)\|<\varepsilon$ for all $t\geq 0$;

\item \textbf{asymptotically stable}, if the trivial solution of system~\eqref{sys} is stable, and there exists a $\sigma>0$ such that $\|x_0\|<\sigma \Rightarrow \lim_{t\rightarrow \infty} \|x(t)\|=0$,
\end{itemize}
where $x(t)=x(t,0,x_0)$ denotes the solution of system~\eqref{sys}.
\end{definition}

To investigate stability of impulsive control system~\eqref{sys}, we define the upper right-hand Dini derivative of a function $V\in \mathcal{V}$ along the trajectory of system~\eqref{sys} as 
\[
\textrm{D}^+ V(x)= \limsup_{h\rightarrow 0^+} \frac{ V(x+hf(x))-V(x) }{h}.
\]
Given a positive constant $R$, we define a set 
\[
\mathcal{B}(R)=\{x\in\mathbb{R}^n : \|x\|< R\},
\]
and then make the following two assumptions for system~\eqref{sys} on $\mathcal{B}(R)$ throughout this paper.
\begin{assumption}\label{assumption1} 
For any $x\in\mathcal{B}(R)$, there exist positive constants $L_1$ and $L_2$ such that 
\[ \|f(x)\| \leq L_1\|x\| \textrm{~~and~~}  \|x+g(x)\| \leq L_2\|x\|.\]

\end{assumption}

\begin{assumption}\label{assumption2} 
There exist functions $V\in\mathcal{V}$, $\alpha_1,\alpha_2\in\mathcal{K}_{\infty}$, positive constants $\mu$ and $\rho$ such that, for any $x\in\mathcal{B}(R)$,
\begin{itemize}
\item[(i)] $\alpha_1(\|x\|)\leq V(x) \leq \alpha_2(\|x\|)$;

\item[(ii)] $\mathrm{D}^+ V(x)\leq \mu V(x)$;

\item[(iii)] if $y$ and $y+g(x)\in\mathcal{B}(R)$, then
\[
V(y+g(x))\leq \rho V(x),
\]
where $y=z(\tau)$, and $z(t)$ is the solution of the following initial value problem (IVP)
\begin{eqnarray}\label{IVP}
\left\{\begin{array}{ll}
\dot{z}(t)=f(z(t)),\cr
z(0)=x.
\end{array}\right.
\end{eqnarray} 
\end{itemize}
\end{assumption}

\begin{remark}\label{Remark0}
{In Assumption~\ref{assumption2}, the behavior of the Lyapunov candidate along the solutions is characterized by two
constants $\mu$ and $\rho$, i.e., the so-called linear rates, which allow easily verifiable conditions on actuation delays to ensure the exponential convergence of the Lyapunov candidate with the proposed event-triggering algorithms in the following two sections. Nevertheless, assumptions with nonlinear rates lead to less conservative sufficient conditions on stability of impulsive control systems. We will investigate along the line of such an extension in our future research.} At each impulse time, we can get from the impulses of system~\eqref{sys} that $x(t_k+\tau)=x((t_k+\tau)^-) + g(x(t_k))$ which corresponds to $y+g(x)$ in Assumption~\ref{assumption2}(iii). State $x((t_k+\tau)^-)$ plays the role of $y$ and is equal to $z(\tau)$, where $z(t)$ is the solution of IVP~\eqref{IVP} with $z(0)=x(t_k)$. Hence, condition~(iii) outlines the relation between the Lyapunov function $V$ at time $t_k$ when the impulsive control signals are updated and the value of $V$ at the impulse time $t_k+\tau$.
\end{remark}

\section{Continuous Event Detection}\label{Sec3}
In this section, we consider event-triggered impulsive control system~\eqref{sys} with the event times $\{t_k\}_{k\in\mathbb{N}}$ determined by an event trigger that is based on continuous sampling of the system states. To be more specific, the event times are determined by
\begin{align}\label{ETM1}
t_{k+1}=\left\{\begin{array}{ll}
\inf\{t\geq 0 : V(x(t)) \geq a e^{-bt} \}, &\mathrm{~if~} k=0\cr
\inf\{t\geq t_k+\tau : V(x(t)) \geq a e^{-bt} \}, &\mathrm{~if~} k\geq 1
\end{array}\right.
\end{align}
where $a$ and $b$ are positive constants. The event trigger~\eqref{ETM1} works as follows. For any initial state $x_0\in\mathbb{R}^n$ so that $V(x_0)<a$, the first event time $t_1$ is the moment when the Lyapunov function $V$ reaches the threshold $a e^{-bt}$. {The states stored in the impulsive controller are updated at time $t_1$.} Due to the actuation delays, the impulsive control is executed to system~\eqref{sys} at time $t=t_1+\tau$ that is an impulse time. The purpose of the impulsive control is to bring down the value of $V$ below the threshold line. Then, the next event time $t_2$ is the time when $V$ reaches the threshold again so that the control input is updated, and the impulse is activated at time $t=t_2+\tau$. This process continues for all $t\geq 0$, and is depicted in Fig.~\ref{CMechanism}. {It can be observed that the impulse times are implicitly defined by~\eqref{ETM1}, and system~\eqref{sys} with different initial conditions will have different sequences of impulse times according to~\eqref{ETM1}.} To ensure asymptotic stability of system~\eqref{sys}, the actuation delay $\tau$ should be upper bounded, and the impulsive input function $g$ needs to be properly designed such that the Lyapunov function $V$ stays below the threshold after each impulse. The following theorem provides some sufficient conditions on asymptotic stability of system~\eqref{sys} with event times determined by~\eqref{ETM1}.

\begin{figure}[!t]
\centering
\includegraphics[width=2.5in]{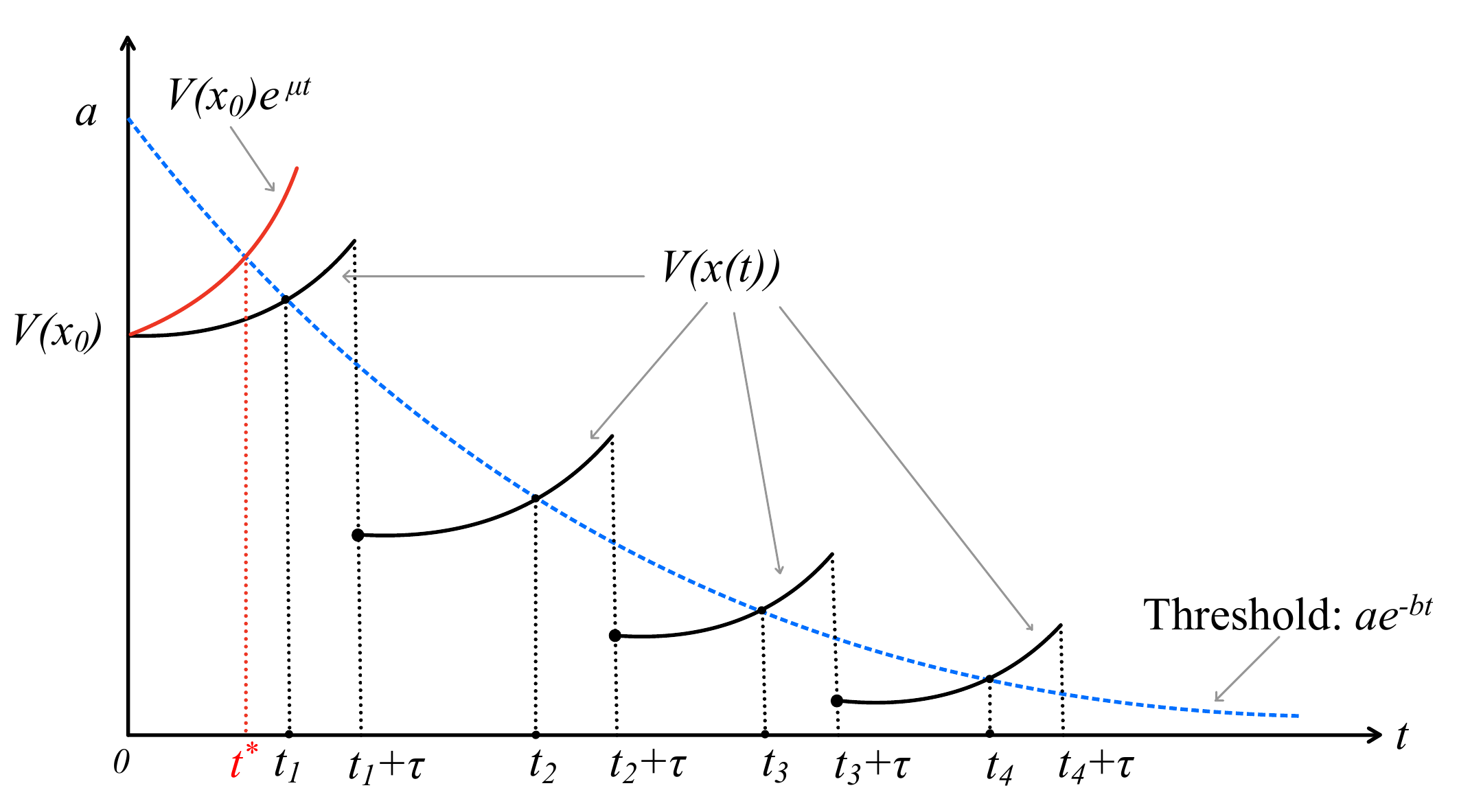}
\caption{Mechanism of event-triggered impulsive control with continuous event detection by~\eqref{ETM1}.}
\label{CMechanism}
\end{figure}

\begin{theorem}\label{Th1}
Consider system~\eqref{sys} with the event times determined by~\eqref{ETM1}, and {suppose that there exists some $R>0$ so that Assumptions~\ref{assumption1} and~\ref{assumption2} hold for all $x\in\mathcal{B}(R)$}. If {$\rho<1$,} $a<\alpha_1(R)$ and $0\leq \tau< \min\{\varepsilon_1,\varepsilon_2\}$ with  {\small
\[
\varepsilon_1= \frac{1}{\mu} \ln\left( \frac{\alpha_1(R)}{a} \right) \textrm{~and~}\varepsilon_2= \frac{1}{b} \ln\left( \frac{1}{\rho} \right)
\]}
satisfy
\begin{equation}\label{inequality1}
L_2+  \sqrt{\frac{\tau L_1 (e^{2L_1\tau}-1)}{2}} <\frac{R}{\alpha^{-1}_1(a)},
\end{equation}
then, for any initial state $x_0\in \mathcal{B}(\alpha^{-1}_2(a))$, the inter-event times $\{t_{k+1}-t_k\}_{k\in\mathbb{N}}$ are lower bounded by
\[
\Gamma:=\tau-\frac{\ln(\rho e^{b\tau})}{b+\mu}> \tau,
\]
that is, $t_{k+1}-t_k\geq \Gamma$ for all $k\in\mathbb{N}$. Moreover, the trivial solution of system~\eqref{sys} is {asymptotically stable}.
\end{theorem}

\begin{proof} {We first assume that $t_k$ is finite for each $k\in \mathbb{N}$}. To guarantee the validity of Assumptions~\ref{assumption1} and~\ref{assumption2}, we need to ensure $x(t)\in \mathcal{B}(R)$ for all $t\geq 0$. {For $t\in [t_0,t_1]$, we can derive from~\eqref{ETM1} that $V(x(t))\leq a e^{-bt}$,
and then
\[\alpha_1(\|x\|)\leq V(x(t))\leq a e^{-bt}\leq a,\]
that is, $\|x\|\leq \alpha^{-1}_1(a)<R$. Hence, $x(t)\in\mathcal{B}(R)$ for $t\in [t_0,t_1]$.} Next, we show $x(t)\in \mathcal{B}(R)$ for all $t\in[t_k,t_k+\tau)$ and $k\in\mathbb{N}$ by contradiction argument. Suppose there exists some $k\in\mathbb{N}$ and $t\in [t_k,t_k+\tau)$ so that $\|x(t)\|\geq R$, then we define 
\[\bar{t}:=\inf\left\{t\in[t_k,t_k+\tau): \|x(t)\|\geq R\right\}.\]
The fact $\alpha_1(\|x(t_k)\|)\leq V(x(t_k))=a e^{-bt_k}$ implies $\|x(t_k)\|\leq \alpha^{-1}_1(ae^{-bt_k})\leq \alpha^{-1}_1(a)<R$. We can conclude from the definition of $\bar{t}$ and the continuity of $x$ on $[t_k,t_k+\tau)$ that $\|x(\bar{t})\|=R$ and $\|x(t)\|<R$ for $t\in[t_k,\bar{t})$. We then get from (ii) of Assumption~\ref{assumption2} that
\begin{align*}
V(x(\bar{t})) \leq V(x(t_k)) e^{\mu(\bar{t}-t_k)} <   a e^{-bt_k}e^{\mu\tau} < a e^{\mu\tau}
\end{align*}
which implies $\alpha_1(\|x(\bar{t})\|)<a e^{\mu\tau}$, that is, $\|x(\bar{t})\|<\alpha^{-1}_1(ae^{\mu\tau})< R$, where we used $\tau <\varepsilon_1=\frac{1}{\mu} \ln\left( \frac{\alpha_1(R)}{a} \right)$. This is a contradiction to the definition of $\bar{t}$. Hence, we conclude that $x(t)\in\mathcal{B}(R)$ for all $t\in[t_k,t_k+\tau)$.

At $t=t_k+\tau$, we have
\begin{align*}
x(t_k+\tau) &= x((t_k+\tau)^-) +g(x(t_k)) \cr
            &= x(t_k)+\int^{t_k+\tau}_{t_k} f(x(s)) \mathrm{d}s +g(x(t_k))
\end{align*}
which implies
\begin{align}\label{y+g(x)}
&~~~ \|x(t_k+\tau)\| \cr
&\leq \|x(t_k)+g(x(t_k))\| + \sqrt{\tau} \left( \int^{t_k+\tau}_{t_k} \|f(x(s))\|^2 \mathrm{d}s \right)^{\frac{1}{2}} \cr
                    &\leq L_2\|x(t_k)\| + \sqrt{\tau} L_1 \left( \int^{t_k+\tau}_{t_k} \|x(s)\|^2 \mathrm{d}s \right)^{\frac{1}{2}} \cr
                    &\leq L_2\|x(t_k)\|  + \sqrt{\tau} L_1 \left( \int^{t_k+\tau}_{t_k} \|x(t_k)\|^2 e^{2L_1(s-t_k)} \mathrm{d}s \right)^{\frac{1}{2}} \cr
                    & = \left( L_2+  \sqrt{\frac{\tau L_1 (e^{2L_1\tau}-1)}{2}} \right) \|x(t_k)\|\cr
                    &\leq \left( L_2+  \sqrt{\frac{\tau L_1 (e^{2L_1\tau}-1)}{2}} \right)  \alpha^{-1}_1(a) \cr
                    &< R,
\end{align}
where we used Schwarz's inequality in the first inequality and Assumption~\ref{assumption1} in the second inequality with the fact $x(t)\in\mathcal{B}(R)$ for all $t\in[t_k,t_k+\tau)$. The third inequality is from the following
\begin{equation}\label{derivative}
\frac{\mathrm{d}}{\mathrm{d}t} \|x(t)\|^2  = \frac{\mathrm{d}}{\mathrm{d}t} x^T(t)x(t) = 2x^T(t)f(x(t)) \leq 2L_1 \|x(t)\|^2,
\end{equation}
where $t\in[t_k,t_k+\tau)$. The last inequality of~\eqref{y+g(x)} follows from~\eqref{inequality1}. We can conclude from~\eqref{y+g(x)} that $x((t_k+\tau)^-) +g(x(t_k))\in\mathcal{B}(R)$ and clearly $x((t_k+\tau)^-)=z(\tau)$ if $z(t)$ is the unique solution of IVP~\eqref{IVP} with initial condition $z(0)=x(t_k)$, which then implies all conditions of~(iii) from Assumption~\ref{assumption2} are satisfied. Hence, Assumption~\ref{assumption2} states that 
\begin{align}\label{afterimpulse}
V(x(t_k+\tau)) & =   V\left(x((t_k+\tau)^-) +g(x(t_k))\right)\cr
               &\leq \rho V(x(t_k)) = \rho a e^{-bt_k} = \rho e^{b\tau} a e^{-b(t_k+\tau)}\cr
               & <   a e^{-b(t_k+\tau)},
\end{align}
which implies that $V$ is below the threshold line at time $t=t_k+\tau$, and $V(x(t))\leq a e^{-bt}$ for $t\in[t_k+\tau,t_{k+1}]$. The last inequality of~\eqref{afterimpulse} was derived from the condition $\tau<\varepsilon_2=\frac{1}{b} \ln\left( \frac{1}{\rho} \right)$. Therefore, we can conclude that $x(t)\in\mathcal{B}(R)$ for all $t\geq 0$.

From~\eqref{afterimpulse} and~(ii) of Assumption~\ref{assumption2}, we have
\begin{align*}
V(t_{k+1}) &\leq V(t_k+\tau) e^{\mu(t_{k+1}-t_k-\tau)},
\end{align*}
then, we can get from the definition of the event times and~\eqref{afterimpulse} that
\begin{align*}
a e^{-b t_{k+1}} \leq \rho e^{b\tau} a e^{-b(t_k+\tau)} e^{\mu(t_{k+1}-t_k-\tau)},
\end{align*}
which implies 
\[
t_{k+1}-t_k\geq \Gamma=\tau-\frac{\ln(\rho e^{b\tau})}{b+\mu},
\]
that is, the inter-event times $\{t_{k+1}-t_k\}_{k\in\mathbb{N}}$ are lower bounded by $\Gamma$ which is bigger than $\tau$ because of $\tau<\varepsilon_2$.

In what follows, we show the attractivity of system~\eqref{sys}. For any $t\geq t_1$, there exists some $k\in\mathbb{N}$ so that one of the following cases occurs.

{Case I:} $t\in[t_k+\tau,t_{k+1})$

According to the definition of the event times, we have
\begin{equation}\label{V(x)1}
V(x(t))< a e^{-bt}.
\end{equation}

{Case II:} $t\in[t_{k},t_{k}+\tau)$

From~(ii) of Assumption~\ref{assumption2}, we get
\begin{align}\label{V(x)2}
V(x(t)) &\leq V(x(t_{k})) e^{\mu(t-t_{k})} \cr
        & =   a e^{-bt_{k}} e^{\mu(t-t_{k})} \cr
        & =   a e^{-bt} e^{(b+\mu)(t-t_{k})} \cr
        & <   a e^{-bt} e^{(b+\mu)\tau}.
\end{align}
From~\eqref{V(x)1},~\eqref{V(x)2}, and~(i) of Assumption~\ref{assumption2}, we conclude $\alpha_1(\|x(t)\|)\leq V(x(t))<a e^{-bt} e^{(b+\mu)\tau}$ for all $t\geq t_1$, which then implies the atractivity of system~\eqref{sys}.

Finally, we show stability of system~\eqref{sys}. According to the definition of the event times, we have $V(x(t))\leq a e^{-bt}$ for $t\in [0,t_1]$. From $x_0\in\mathcal{B}(\alpha^{-1}_2(a))$ and~(ii) of Assumption~\ref{assumption2}, we have $V(x_0)\leq \alpha_2(\|x_0\|)< a$, and there exists a unique $t^*>0$ such that $V(x_0)e^{\mu t^*}=a e^{-b t^*}$, which means 
\[
t^*=\frac{1}{\mu+b} \ln\left( \frac{a}{V(x_0)} \right).
\]
More explicitly, the graph of $V(x_0)e^{\mu t}$ intersects with the threshold line $a e^{-bt}$ at time $t=t^*$, and such an intersection is unique due to the fact that $V(x_0)e^{\mu t}$ is strictly increasing and $a e^{-bt}$ is strictly decreasing. According to the trigger~\eqref{ETM1}, we can see that $t_1$ is the time when $V(x(t))$ reaches the threshold line for the first time. Therefore, the fact that $V(x(t))\leq V(x_0)e^{\mu t}$ implies $t^*\leq t_1$.
See Fig.~\ref{CMechanism} for the demonstration of $t^*$ and the relation between $V(x(t))$ and $V(x_0)e^{\mu t}$ on the interval $[0,t_1+\tau)$. 

If $t^*<t_1$, we then can make the following two conclusions on $[0,t_1)$.
\begin{itemize}
\item When $t\in[0,t^*]$, we have 
\[
V(x(t))\leq V(x_0)e^{\mu t}\leq V(x_0) e^{\mu t^*}=a e^{-b t^*}.
\]

\item  When $t\in[t^*,t_1)$, we have $V(x(t))\leq a e^{-bt}\leq a e^{-b t^*}$.

\end{itemize}
We can derive from the above two scenarios that
\begin{equation}\label{stability1.1}
V(x(t))\leq a e^{-b t^*} \textrm{~for~all~} t\in[0,t_1).
\end{equation}
If $t^*=t_1$ then $V(x(t))\leq V(x_0)e^{\mu t}\leq V(x_0)e^{\mu t^*}= a e^{-b t^*}$ for all $t\in [0,t_1)$, which implies~\eqref{stability1.1} is also true.

When $t\geq t_1$, there is a $k\in\mathbb{N}$ so that $t\in[t_k,t_{k+1})$. If $t\in [t_k, t_k+\tau)$, we get from~\eqref{V(x)2} that $V(x(t))\leq a e^{-b t_k} e^{\mu \tau}$. If $t\in [t_k+\tau,t_{k+1})$, we have from~\eqref{V(x)1} that $V(x(t))\leq a e^{-b t}\leq a e^{-b t_k}$. Then, we can conclude that
\begin{align}\label{stability1.2}
V(x(t)) \leq a e^{-b t_k} e^{\mu \tau} \leq a e^{-b t_1} e^{\mu \tau} \leq a e^{-b t^*} e^{\mu \tau}
\end{align}
for $t\geq t_1$. Hence, we derive from~\eqref{stability1.1} and~\eqref{stability1.2} that 
\begin{align}\label{stability1.3}
V(x(t)) \leq a e^{-b t^*} e^{\mu \tau} \textrm{~for~all~} t\geq 0.
\end{align}
Applying~(i) of Assumption~\ref{assumption2} to~\eqref{stability1.3} with the definition of $t^*$ yields 
\begin{align}\label{stability1.4}
\|x(t)\| \leq \alpha^{-1}_1\left(a e^{-b t^*} e^{\mu \tau}\right) =   \alpha^{-1}_1\left(a e^{-\frac{b}{\mu+b} \ln\left( \frac{a}{V(x_0)} \right)} e^{\mu \tau}\right)
\end{align}
for all $t\geq 0$. Therefore, for any $\varepsilon>0$, there exists a 
\[
\sigma=\alpha^{-1}_2\left( a \left(  \frac{\alpha_1(\varepsilon)}{a e^{\mu\tau}}  \right)^{\frac{\mu+b}{b}} \right)
\]
such that, for any $\|x_0\| < \min\left\{\sigma,\alpha_2^{-1}(a)\right\}$, we can derive from~\eqref{stability1.4} that $\|x(t)\|<\varepsilon$ for all $t\geq 0$. This shows stability of the trivial solution of system~\eqref{sys}.

{If there exists some $k\in \mathbb{N}$ such that 
\[ V(x(t)) < a e^{-bt} \textrm{ for all } t\geq t_{k}+\tau,\]
that is, $t_{k+1}=\infty$, then the impulses will not be triggered after time $t_{k}+\tau$, and attractivity follows naturally for this scenario. From the above discussion, we can conclude similarly that $x(t)\in\mathcal{B}(R)$ and stability of  system~\eqref{sys}. }
\end{proof}

\begin{remark}\label{Remark1}
It can be seen from the event trigger~\eqref{ETM1} that the inter-event times $\{t_{k+1}-t_k\}_{k\in\mathbb{N}}$ are lower bounded by $\tau$ provided $\tau>0$, and there is no need to monitor the system states over the time interval $(t_k, t_k+\tau)$. We can also observe that the event trigger~\eqref{ETM1} requires continuous monitoring of the states over the time interval $[t_k+\tau,t_{k+1}]$. In Theorem~\ref{Th1}, we derived a larger bound $\Gamma>\tau$ {(because of $\ln(\rho e^{b\tau})<0$)} which is also applicable to system~\eqref{sys} without actuation delays, that is, $\tau=0$. Therefore, system~\eqref{sys} does not exhibit Zeno behavior that is a phenomenon of the event-triggered control system triggering infinitely many events over a finite time interval. {In the area of impulsive differential equations, system~\eqref{sys} is said to be absent from pulse phenomenon} which is described as the system experiencing an infinite number of impulses in a finite amount of time.
\end{remark}
{
\begin{remark}
To ensure stability of system~\eqref{sys}, two types of conditions on $\tau$ are included in Theorem~\ref{Th1}: I. $0\leq \tau<\varepsilon_1$ with~\eqref{inequality1} satisfied; II. $\tau<\varepsilon_2$. The type I condition guarantees that the system trajectory stays in $\mathcal{B}(R)$ for all $t$, so that Assumptions~\ref{assumption1} and~\ref{assumption2} can be applied, and the type II condition ensures the validity of the event-triggering algorithm with event times determined by~\eqref{ETM1}. If $f$ and $g$ are globally Lipschitz, i.e., $R=\infty$, then 
both $0\leq \tau<\varepsilon_1$ and~\eqref{inequality1} hold for all $\tau$, and the only requirement on the actuation delay is $\tau<\varepsilon_2$, which implies $\rho e^{b\tau}<1$. Large delay $\tau$ allows the Lyapunov function to go over and deviate far from the threshold, then big jump of the Lyapunov function at each impulse time is expected such that the Lyapunov function can be smaller than the threshold after the impulse. For $f$ and/or $g$ being locally Lipschitz on $\mathcal{B}(R)$, large $R$ may lead to large $L_1$ and $L_2$, but the largest admissible delay $\tau$ is not directly conclusive from~\eqref{inequality1}. More specifically, consider the scalar system~\eqref{sys} with $f(x)=g(x)=x^3$. Both functions satisfy Assumption II.1 on $\mathcal{B}(R)$ with $L_1=L_2=R^2$ for arbitrary $R>0$. We then can derive from~\eqref{inequality1} that large $R$ allows small admissible $\tau$. It is worth mentioning that the continuous dynamics of $\dot{x}=x^3$ has finite escape time, that is, the nontrivial state goes to infinity in finite time. Nevertheless, the type I condition on $\tau$ in Theorem~\ref{Th1} excludes the finite escape time from the impulsive scalar system by enforcing the state trajectory to stay in $\mathcal{B}(R)$. Hence, global existence of the solution is guaranteed. The type II condition then ensures asymptotic stability. On the other hand, consider another scalar system with $f(x)=cx$ and $g(x)=x^{3/2}$, where $c$ is a positive constant. Then, $f$ is globally Lipschitz with $L_1=c$ and $g$ satisfies Assumption~II.1 on $\mathcal{B}(R)$ with $L_2=\sqrt{R}$ for any $R>0$. We conclude from~\eqref{inequality1} that large enough $R$ could allow large admissible $\tau$. It can be seen that the dependence of $\tau$ on $R$ also relies on the change of $L_1$ and/or $L_2$.
\end{remark}
}

\section{Periodic Event Detection}\label{Sec4}
In this section, we consider event-triggered impulsive control for system~\eqref{sys} with event times $\{t_k\}_{k\in\mathbb{N}}$ determined by an event trigger that relies only on periodic sampling of the system states. Suppose $\delta>0$ is the sampling period to be determined, and the system states are only measured at the moments $t=j\delta$ for $j\in\mathbb{N}$. We propose an event trigger with event times determined as 
\begin{align}\label{ETM2}
t_{k+1}=\left\{\begin{array}{ll}
\inf\{j\delta\geq 0 : V(x(j\delta)) \geq a e^{-bj\delta} \}, &\mathrm{~if~} k=0\cr
\inf\{j\delta\geq t_k+\tau : V(x(j\delta)) \geq a e^{-bj\delta} \}, &\mathrm{~if~} k\geq 1
\end{array}\right.
\end{align}
where $j\in\mathbb{N}$, and $a,b$ are positive constants. At the sampling times $\{j\delta\}_{j\in\mathbb{N}}$, system~\eqref{sys} evaluates decisions about whether to trigger the event or not, according to event trigger~\eqref{ETM2}. Identically to the mechanism of the trigger~\eqref{ETM1}, an event occurs when the Lyapunov function $V$ surpasses the threshold $a e^{-bt}$. Nevertheless, due to the periodic state sampling, the event times determined by~\eqref{ETM2} may not be the moments when $V$ exactly reaches the threshold line and could possibly be some times after these moments.
See Fig.~\ref{PMechanism} for a demonstration. The objective of each impulsive control input is to bring the value of $V$ down below the threshold. Theorem~\ref{Th2} establishes such conditions on the impulse function $g$ and some sufficient conditions on the actuation delay $\tau$ and the sampling period $\delta$ to guarantee asymptotic stability of system~\eqref{sys}.

\begin{figure}[!t]
\centering
\includegraphics[width=2.8in]{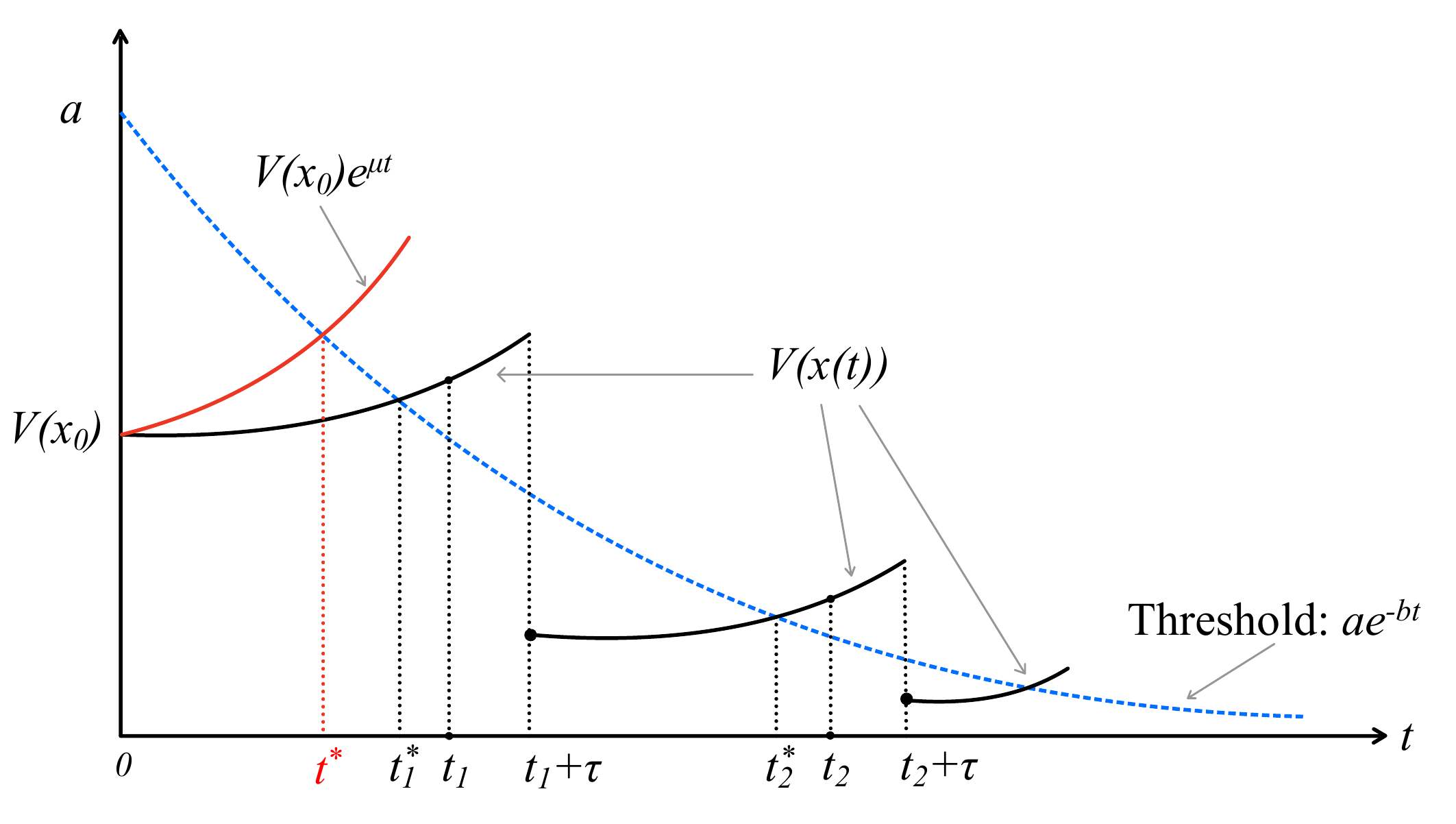}
\caption{Mechanism of event-triggered impulsive control with periodic event detection by~\eqref{ETM2}.}
\label{PMechanism}
\end{figure}

\begin{theorem}\label{Th2}
Consider system~\eqref{sys} with the event times determined by~\eqref{ETM2}, and {suppose that there exists some $R>0$ so that Assumptions~\ref{assumption1} and~\ref{assumption2} hold for all $x\in\mathcal{B}(R)$}. If $\mu$, $\rho$, $a$, $\tau$ and $\delta$ satisfy the following inequalities
\begin{equation}\label{inequality2.1}
\max\left\{e^{(\mu+b)\delta}, e^{\mu(\tau+\delta)}\right\}\leq \frac{\alpha_1(R)}{a},
\end{equation}
\begin{equation}\label{inequality2.3}
\rho e^{b(\tau+\delta)} e^{\mu\delta} < 1,
\end{equation}
and~\eqref{inequality1} holds with constant $a$ replaced by $\bar{a}:=a e^{(\mu+b)\delta}$, then, for any initial state $x_0\in \mathcal{B}(\alpha^{-1}_2(a))$, the inter-event times $\{t_{k+1}-t_k\}_{k\in\mathbb{N}}$ are bigger than
\[
\Gamma:=\tau-\frac{\ln\left( \rho e^{b(\tau+\delta)} e^{\mu\delta} \right)}{b+\mu} >\tau.
\]
Moreover, the trivial solution of system~\eqref{sys} is asymptotically stable.
\end{theorem}

\begin{proof} {Without loss of generality, we assume that $t_k$ is finite for each $k\in \mathbb{N}$}. 
From the fact $x_0\in \mathcal{B}(\alpha^{-1}_2(a))$ and the definition of the event times~\eqref{ETM2}, we conclude that there exists a $t^*_1\in (t_1-\delta,t_1]$ such that $V(x(t^*_1))=a e^{-bt^*_1}$ and $V(x(t))\geq ae^{-bt}$ for $t\in[t^*_1,t_1]$, then $\|x(t^*_1)\|\leq \alpha^{-1}_1(a e^{-bt^*_1})< R$. See Fig.~\ref{PMechanism} for a demonstration of $t^*_1$. Since system states are only measured at periodic sampling times, it is not guaranteed for $V(x(t))<a e^{-bt}$ when $t$ is not a sampling moment. This is different from the proof of Theorem~\ref{Th1}. Thus, we need to bound $V(x(t))$ on $[0,t^*_1]$. To do so, we consider an arbitrary $t\in[0,t^*_1]$, then there is some $j\in\mathbb{N}$ so that $t\in[(j-1)\delta,j\delta]$. Since $t_1$ is the first sampling moment so that $V(x(t))> a e^{-bt}$, we have $V(x(t))\leq a e^{-bt}$ at $t=(j-1)\delta$. Then, for $t\in[0,t^*_1]$, we obtain
\begin{align}\label{overthreshold}
V(x(t)) &\leq V(x((j-1)\delta)) e^{\mu(t-(j-1)\delta)} \cr
        &\leq a e^{-b(j-1)\delta} e^{\mu\delta} \cr
        & =   a e^{\mu\delta} e^{b(t-(j-1)\delta)} e^{-bt} \cr
        &\leq a e^{(\mu+b)\delta} e^{-bt},
\end{align}
which implies 
\begin{align*}
\alpha_1(\|x(t)\|)\leq V(x(t)) \leq a e^{(\mu+b)\delta} e^{-bt} \leq a e^{(\mu+b)\delta},
\end{align*}
that is, $\|x(t)\|\leq \alpha^{-1}_1(a e^{(\mu+b)\delta})<R$ where we used~\eqref{inequality2.1}. Hence, $x(t)\in\mathcal{B}(R)$ for $t\in[0,t^*_1]$.

In what follows, we shall show $x(t)\in\mathcal{B}(R)$ for all $t\in[t^*_1,t_1+\tau)$ by contradiction argument. Suppose, on the contrary, that there exists some $t\in [t^*_1,t_1+\tau)$ so that $\|x(t)\|\geq R$, then we define 
\[\tilde{t}:=\inf\{t\in[t^*_1,t_1+\tau): \|x(t)\|\geq R\},\]
which implies $\|x(\tilde{t})\|=R$ and $x(t)<R$ for $t\in [t^*_1,\tilde{t})$. From~(i),(ii) of Assumption~\ref{assumption2} and the fact $t_1-\delta<t^*_1<\tilde{t}<t_1+\tau$, we have
\begin{align*}
\alpha_1(\|x(\tilde{t})\|)\leq V(x(\tilde{t})) \leq V(x(t^*_1)) e^{\mu(\tilde{t}-t^*_1)}  =   a e^{-bt^*_1} e^{\mu(\tilde{t}-t^*_1)} < a e^{\mu(\tau+\delta)},
\end{align*}
that is, $\|x(\tilde{t})\|<\alpha^{-1}_1\left(a e^{\mu(\tau+\delta)}\right)<R$ where we used inequality~\eqref{inequality2.1}. This is a contradiction to the definition of $\tilde{t}$. Therefore, $x(t)\in\mathcal{B}(R)$ for all $t\in[t^*_1,t_1+\tau)$.

Similarly to the discussions in the proof of Theorem~\ref{Th1}, we can prove $x((t_1+\tau)^-) +g(x(t_1))\in\mathcal{B}(R)$. Then, Assumption~\ref{assumption2} and~\eqref{inequality2.3} conclude that 
\begin{align}\label{afterimpulse2}
V(x(t_1+\tau)) & =   V\left(x((t_1+\tau)^-) +g(x(t_1))\right)\cr
               &\leq \rho V(x(t_1)) \cr
               &\leq \rho V(x(t^*_1)) e^{\mu(t_1-t^*_1)} \cr
               &\leq \rho a e^{-bt^*_1} e^{\mu\delta} \cr
               & <   \rho a e^{-b(t_1-\delta)} e^{\mu\delta} \cr
               & =   \rho e^{b(\tau+\delta)} e^{\mu\delta} a e^{-b(t_1+\tau)} \cr
               & <   a e^{-b(t_1+\tau)},
\end{align}
that is, $V(x(t_1+\tau))< a e^{-b(t_1+\tau)}$. Thus, condition~\eqref{inequality2.3} enforces $V$ at the impulse time $t_1+\tau$ to stay underneath the threshold (see Fig.~\ref{PMechanism} for an illustration). From the continuity of system~\eqref{sys} on $[t_1+\tau,t_2]$, we can see that there exists a $t^*_2\in (t_1+\tau,t_2]$ such that $V(x(t^*_2))=a e^{-b t^*_2}$ and $V(x(t))\geq a e^{-bt}$ for $t\in[t^*_2,t_2]$. Similarly to the discussion of~\eqref{overthreshold}, we get $V(x(t))\leq a e^{(\mu+b)\delta} e^{-bt}$ for $t\in [t_1+\tau,t^*_2]$. From the definition of $t_2$ in~\eqref{ETM2} and inequality~\eqref{afterimpulse2}, we have that
\begin{align*}
a e^{-bt_2} &\leq V(x(t_2)) \cr
            &\leq V(x(t_1+\tau)) e^{\mu(t_2-t_1-\tau)} \cr
            & <   \rho e^{b(\tau+\delta)} e^{\mu\delta} a e^{-b(t_1+\tau)} e^{\mu(t_2-t_1-\tau)},
\end{align*}
which implies
\[
t_2-t_1> \Gamma=\tau-\frac{\ln\left( \rho e^{b(\tau+\delta)} e^{\mu\delta} \right)}{b+\mu}.
\]

Repeating the above discussion over the interval $[t_{k-1}+\tau,t_{k}+\tau)$ for each $k\geq 2$, we can show that 
\begin{itemize}
\item there exists a $t^*_k\in(t_k-\delta,t_k]$ such that $t^*_k>t_{k-1}+\tau$, $V(x(t^*_k))=a e^{-bt^*_k}$, and $V(x(t))\geq a e^{-bt}$ for $t\in [t^*_k,t_k]$;

\item when $t\in[t_{k-1}+\tau,t^*_k)$, $x(t)\in \mathcal{B}(R)$ and $V(x(t))\leq a e^{(\mu+b)\delta} e^{-bt}$;

\item when $t\in[t^*_k,t_k+\tau)$, $x(t)\in \mathcal{B}(R)$ and
\begin{align}\label{summary1}
V(x(t)) &\leq V(x(t^*_k)) e^{\mu(t-t^*_k)} =   a e^{-bt}e^{(b+\mu)(t-t^*_k)} \cr
        & <   a e^{-bt}e^{(b+\mu)(\tau+\delta)};
\end{align}

\item for any $k\in\mathbb{N}$, we have $x((t_k+\tau)^-) +g(x(t_k))\in\mathcal{B}(R)$, then similarly to the discussion of~\eqref{afterimpulse2},
\begin{equation}\label{afterimpulse3}
V(x(t_k+\tau))  =   V\left(x((t_k+\tau)^-) +g(x(t_k))\right) < a e^{-b(t_k+\tau)};
\end{equation}

\item $t_{k+1}-t_k>\Gamma$.

\end{itemize}


Therefore, from all the above discussions, we have 
\[ 
V(x(t))< a e^{-bt}e^{(b+\mu)(\tau+\delta)} \textrm{~for~} t\geq 0,
\] 
which with~(i) of Assumption~\ref{assumption2} implies attractivity of system~\eqref{sys}. Replacing $a$ by $\bar{a}$ and $\tau$ by $\tau +\delta$ in~\eqref{stability1.2},~\eqref{stability1.3},~\eqref{stability1.4} and the definition of $\sigma$ in the proof of Theorem~\ref{Th1}, stability of system~\eqref{sys} with trigger~\eqref{ETM2} can be established identically, and thus is omitted.
\end{proof}

\begin{remark}\label{Remark2}
Other than $\Gamma$, we can tell from the event-triggering mechanism that $\delta$ is also a lower bound of the inter-impulse times $\{t_{k+1}+\tau-(t_k+\tau)\}_{k\in\mathbb{N}}$ from the mechanism of trigger~\eqref{ETM2}, that is, $t_{k+1}-t_k\geq \delta$ for $k\in\mathbb{N}$. Therefore, we conclude 
\begin{align*}
t_{k+1}-t_k\left\{\begin{array}{ll}
> \Gamma, &\mathrm{~if~} \Gamma \geq \delta\cr
\geq \delta, &\mathrm{~if~} \delta  >   \Gamma
\end{array}\right.
\end{align*}
\end{remark}

\begin{remark}\label{Remark3}
It is worth noting that the state sampling times are determined by the sampling period $\delta$ and independent of the event times defined by~\eqref{ETM2}, and only the states sampled after the time $t_k+\tau$ are utilized to determine the event time $t_{k+1}$. Therefore, to improve the sampling efficiency, we propose the following event trigger
\begin{align}\label{ETM3}
t_{k+1}=\left\{\begin{array}{ll}
\inf\{ j\delta\geq 0 : V(x(j\delta)) \geq a e^{-bj\delta} \}, &\mathrm{~if~} k=0\cr
\inf\{ t=t_k+\tau+j\delta : V(x(t)) \geq a e^{-bt} \}, &\mathrm{~if~} k\geq 1
\end{array}\right.
\end{align}
where $j\in\mathbb{N}$, $\delta>0$ is the sampling period, and $a,b$ are positive constants. It can be observed from~\eqref{ETM3} that the sampling process starts at each impulse time $t_k+\tau$ and ends when the next event time $t_{k+1}$ is determined, and this sampling process repeats at every impulse time. Compared with~\eqref{ETM2}, the event trigger~\eqref{ETM3} is more efficient in the sense that the system states are measured periodically in the time interval $[t_k+\tau,t_{k+1}]$ for $k\in\mathbb{N}$ instead of the entire time span. Theorem~\ref{Th2} with its proof still holds for system~\eqref{sys} with event times determined by~\eqref{ETM3}. Other than the lower bound $\Gamma$ derived in Theorem~\ref{Th2}, we can get another lower bound of the inter-event times that is $\tau+\delta$, from the mechanism of the trigger~\eqref{ETM3}. Hence, we summarize the lower bound of the inter-event times as follows:
\begin{align*}
t_{k+1}-t_k\left\{\begin{array}{ll}
> \Gamma, &\mathrm{~if~} \Gamma \geq \tau+\delta\cr
\geq \tau+\delta, &\mathrm{~if~} \tau+\delta  >   \Gamma
\end{array}\right.
\end{align*}
\end{remark}

\section{An Illustrative Example}\label{Sec5}
In this section, we consider the following impulsive control system to illustrate our theoretical results:
\begin{eqnarray}\label{example-sys}
\left\{\begin{array}{ll}
\dot{x}(t)=A x(t) + B h(x(t)), ~~t\not=t_k+\tau,\cr
\Delta x(t_k+\tau)=C(t_k) x(t_k),~~k\in\mathbb{N},\cr
x(0)=x_0,
\end{array}\right.
\end{eqnarray}
where $x=(x_1,x_2,x_3)^{\textrm{T}}\in\mathbb{R}^3$, $h(x)=(q(x_1),q(x_2),q(x_3))^\textrm{T}$ with function $q$ defined as $q(z)=\frac{1}{2}\left( |z+1|-|z-1| \right)$ for $z\in\mathbb{R}$, and the matrices 
\[
A=\begin{bmatrix}
   -1 & 0 & 0 \\
    0 & -1 & 0 \\
    0 & 0 & -1 \\
\end{bmatrix},~ B=\begin{bmatrix}
    1.25 & -3.2 & -3.2 \\
    -3.2 & ~1.1 & -4.4 \\
    -3.2 & ~4.4 & ~1.0
\end{bmatrix}
\]
and
\[
C(t)=\begin{bmatrix}
   -1 & 0 & -\frac{2}{5}\cos(\pi t) \\
    0 & -1-\frac{2}{5}\sin(\pi t) & 0 \\
    -\frac{2}{5}\sin(2t) & 0 & -1 \\
\end{bmatrix}
\]
for $t\in\mathbb{R}$. The actuation delay is $\tau\geq 0$, and the event times $\{t_k\}_{k\in\mathbb{N}}$ are to be determined according to the event triggers~\eqref{ETM1},~\eqref{ETM2},~and~\eqref{ETM3}, respectively.
It has been shown in~\cite{WH-BZ-QLH-FQ-JK-JC:2017} that system~\eqref{example-sys} with $C\equiv 0$ exhibits chaotic behavior.

From the dynamics of system~\eqref{example-sys}, we have
\[
\|Ax+Bh(x)\|\leq \|Ax\| +\|B\| \ \|h(x)\| \leq (\|A\|+\|B\|)\|x\|
\]
and
\[
\|x+C(t)x\|\leq \|I+C(t)\| \ \|x\| \leq \frac{2}{5}\|x\| \textrm{~for all~} k\in\mathbb{N} \textrm{~and~} t\in\mathbb{R},
\]
thus Assumption~\ref{assumption1} holds for system~\eqref{example-sys} with $L_1=\|A\|+\|B\|\approx 8.010$ and $L_2=2/5$.

Regarding Assumption~\ref{assumption2}, we select Lyapunov candidate $V(x(t))=\|x(t)\|$. Then, (i) of Assumption~\ref{assumption2} holds with $\alpha_1(\|x\|)=\alpha_2(\|x\|)=\|x\|$. When $t$ is not an impulse time and $\|x(t)\|\not=0$, we get from the continuous dynamics of system~\eqref{example-sys} that
\begin{align}\label{dV}
\frac{\textrm{d}}{\textrm{d}t} \|x(t)\| = \frac{\textrm{d}}{\textrm{d}t} \sqrt{x^{\textrm{T}}x} = \frac{x^{\textrm{T}}\dot{x}}{\|x\|} \leq L_1 \|x\|.
\end{align}
If $\|x(t)\|=0$ at some $t$, then $x$ stays at $0$ for any time after time $t$, which means $\frac{\textrm{d}}{\textrm{d}t} \|x(t)\| \leq L_1 \|x\|$ also holds for this scenario. Hence,~(ii) of Assumption~\ref{assumption2} is satisfied with $\mu=L_1$. From~\eqref{y+g(x)}, we have
\[
\|x((t_k+\tau)^-) + g(x(t_k))\|\leq \rho \|x(t_k)\|
\]
with 
\[
\rho=L_2+  \sqrt{\frac{\tau L_1 (e^{2L_1\tau}-1)}{2}},
\] 
which implies~(iii) of Assumption~\ref{assumption2} holds. Therefore, Assumption~\ref{assumption2} is true for system~\eqref{example-sys} with $V(x)=\|x\|$. It can be observed that both Assumptions~\ref{assumption1} and~\ref{assumption2} hold globally, that is, $R=\infty$ under these assumptions.

If
\begin{equation}\label{example-condtion1}
 \left( L_2+  \sqrt{\frac{\tau L_1 (e^{2L_1\tau}-1)}{2}} \right) e^{b\tau}<1
\end{equation}
then all the conditions of Theorem~\ref{Th1} are satisfied. In the simulation, we let $a=0.294$, $b=0.1$, and $\tau=0.05$, so that~\eqref{example-condtion1} holds. Then, we can compute $\rho\approx 0.896$ and then $\Gamma\approx 0.063$, that is, $t_{k+1}-t_k \geq 0.063$ for all $k\in\mathbb{N}$. See Fig.~\ref{CED} for the simulation of system~\eqref{example-sys} with event times defined by~\eqref{ETM1}.

\begin{figure}[!t]
\centering
\includegraphics[width=2.8in]{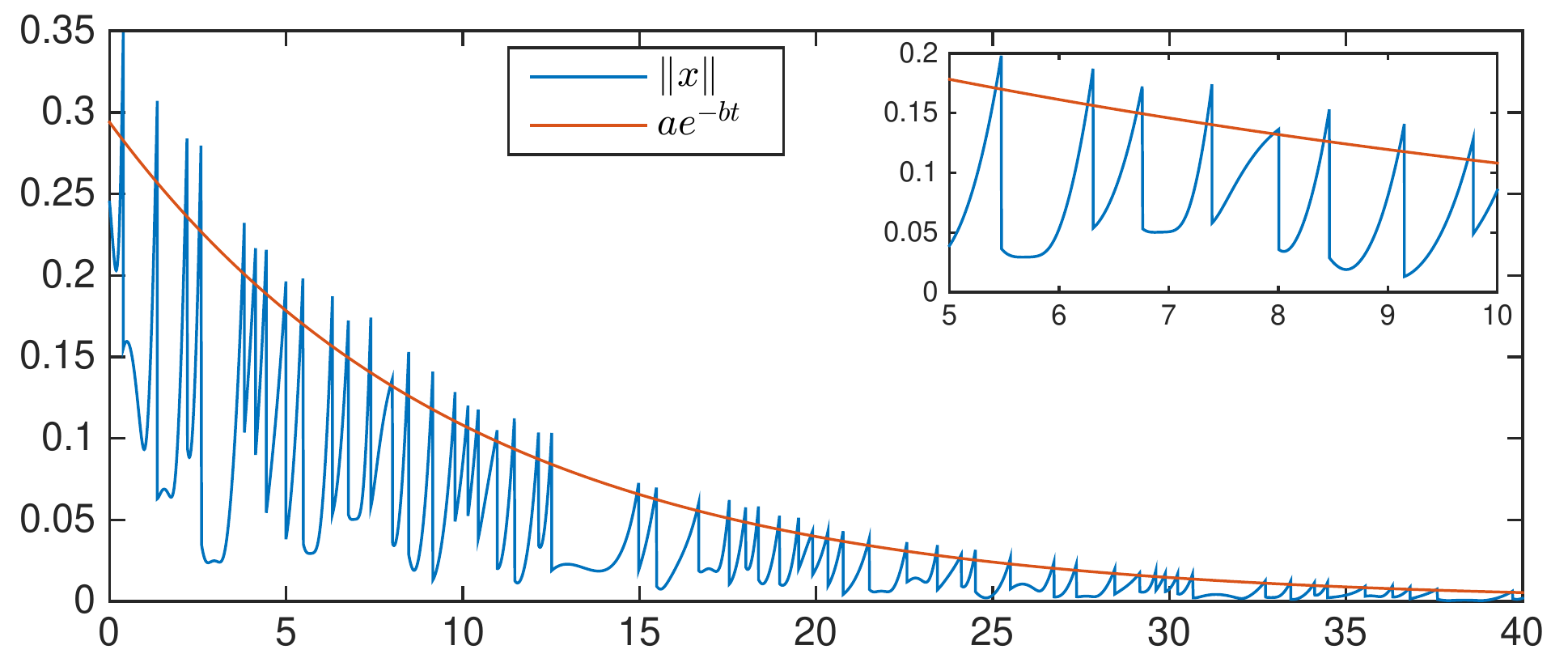}
\caption{Evolution of $\|x\|$ for system~\eqref{example-sys} with $x_0=(0.1, 0.2, -0.1)^{\textrm{T}}$ under continuous event detection.}
\label{CED}
\end{figure}

If
\begin{equation}\label{example-condtion2}
 \left( L_2+  \sqrt{\frac{\tau L_1 (e^{2L_1\tau}-1)}{2}} \right) e^{b(\tau+\delta)} e^{L_1 \tau}<1
\end{equation}
then all the conditions of Theorem~\ref{Th2} are satisfied. In the simulation, we let $a=0.294$, $b=0.1$, $\tau=0.04$, and $\delta=0.015$, so that~\eqref{example-condtion2} holds. We can calculate $\rho\approx 0.779$ and then $\Gamma\approx 0.055\geq \tau+\delta$. Therefore, we get from Theorem~\ref{Th2} and Remark~\ref{Remark2} that the inter-event times are bigger than $0.055$. See Fig.~\ref{PED} for the simulations of system~\eqref{example-sys} with two types of periodic event detection.

\begin{figure}[!t]
\centering
\subfigure[Simulation with event trigger~\eqref{ETM2}.]{\label{PEDa}
\includegraphics[width=2.8in]{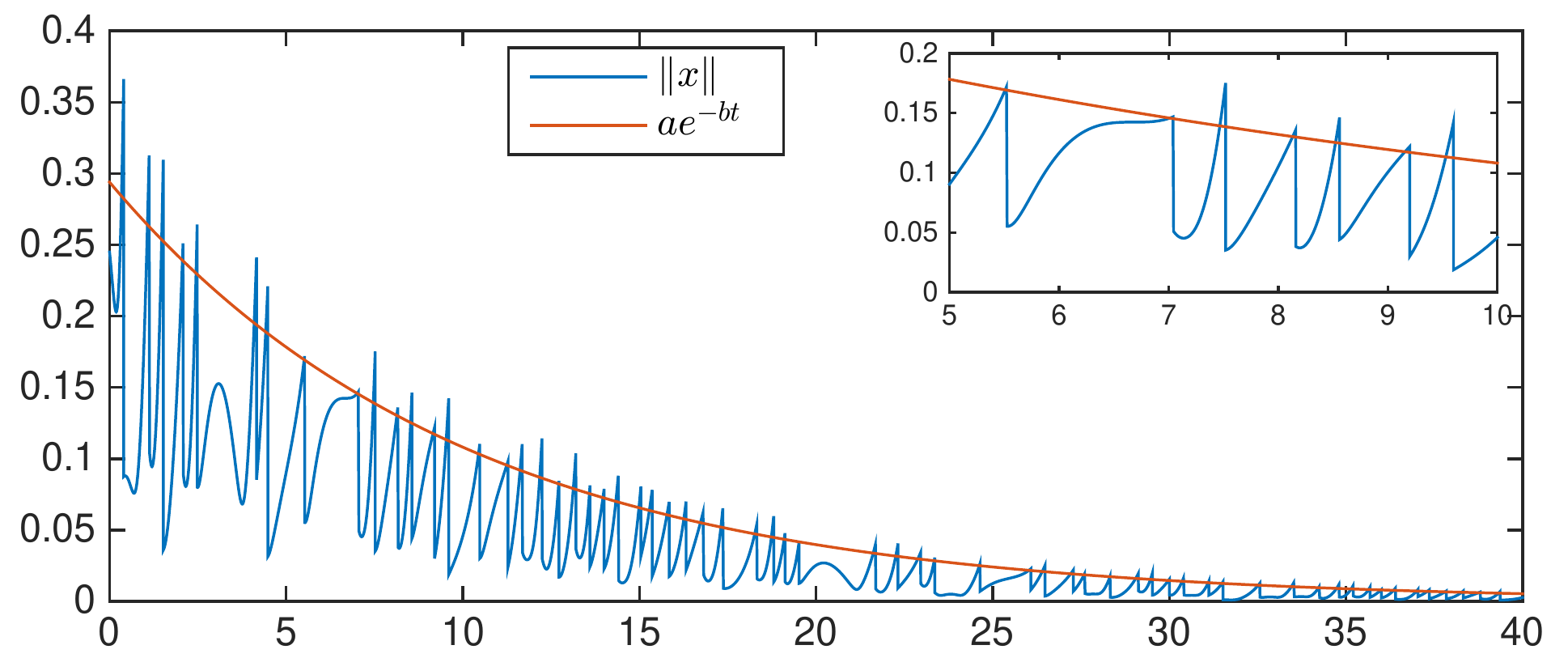}}
\subfigure[Simulation with event trigger~\eqref{ETM3}.]{\label{PEDb}
\includegraphics[width=2.8in]{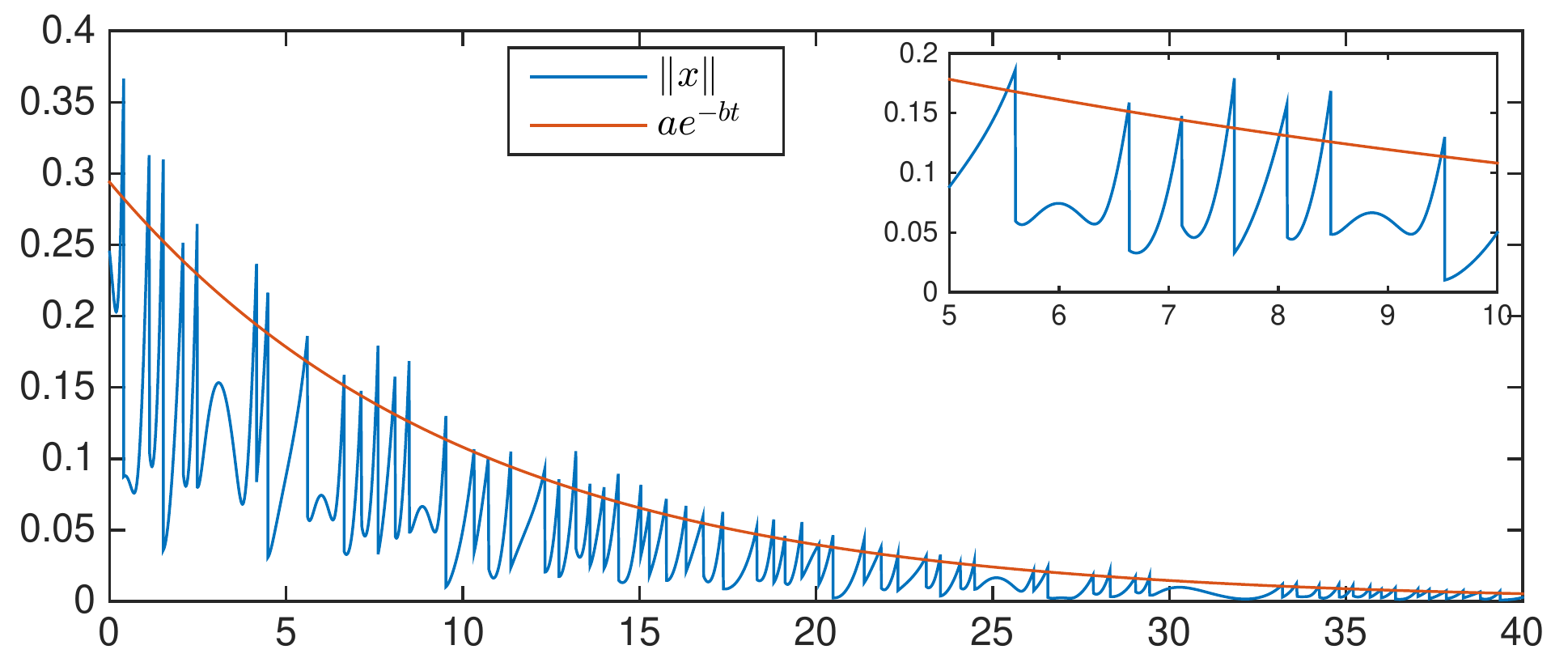}}
\caption{Evolution of $\|x\|$ for system~\eqref{example-sys} with $x_0=(0.1, 0.2, -0.1)^{\textrm{T}}$ under periodic event detection.}
\label{PED}
\end{figure}

It can be seen from Figs.~\ref{CED} and~\ref{PED} that system~\eqref{sys} is asymptotically stable. Nevertheless, due to the existence of the actuation delay $\tau$ or (and) the sampling period $\delta$, the state norm $\|x\|$ does not always stay below the threshold line. It can also be observed that the inter-impulse periods are lower bounded, and the pulse phenomenon is avoided. The vertical line segments correspond to the impulse times and the impulses in our simulations.

To compare with the time-triggered impulsive control, we apply a recent result in~\cite{KZ-EB:2021} to system~\eqref{example-sys}. Let $s_k=t_k+\tau$ for $k\in\mathbb{N}$, then we can rewrite~\eqref{example-sys} in the form of a system with delayed impulses
\begin{eqnarray}\label{example-sys1}
\left\{\begin{array}{ll}
\dot{x}(t)=A x(t) + B h(x(t)), ~~t\not=s_k,\cr
\Delta x(s_k)=C(s_k-\tau) x(s_k-\tau),~~k\in\mathbb{N},\cr
x(0)=x_0,
\end{array}\right.
\end{eqnarray}
where we assume $s_1-\tau\geq 0$ so that the system can be driven by the initial condition in~\eqref{example-sys1}. 

With the Lyapunov function $V(x)=\|x\|$, we conclude from~\eqref{dV} and~\eqref{y+g(x)} that
\[
\dot{V}(x)\leq L_1 V(x) \textrm{~and~} V(x(s_k))\leq \rho V(x(s_k-\tau)).
\]
Let $N(t,s)$ represent the number of impulses in the half-closed interval $(s,t]$. Then, we get from the main result of~\cite{KZ-EB:2021} that if there exist positive constants $\lambda$, $\zeta$ and $\sigma$ such that $\rho e^{\sigma N(s_k,s_k-\tau)}\leq 1$ for all $k\in\mathbb{N}$ and 
\begin{equation}\label{ADT}
N(t,s)\geq \frac{t-s}{T^*}-\frac{\zeta}{\sigma} \textrm{~for~} t>s\geq 0,
\end{equation}
where $T^*=\sigma/(L_1+\lambda)$, then system~\eqref{example-sys1} is asymptotically stable. Inequality~\eqref{ADT} is called a reverse average dwell-time (ADT) condition which requires that there exists at least one impulse per interval of length $T^*$ on average (see~\cite{KZ-EB:2021} for the detailed discussions). We can tell from~\cite{KZ-EB:2021} that $\lambda$ corresponds to the exponential convergence rate of $V$. Therefore, we consider $\lambda=b=0.1$, so that both the event-triggered impulsive control and time-triggered impulsive control share the same convergence rate for $V$. Here, we also consider the same delay $\tau=0.05$, so that $\rho\approx 0.896$ is the same as that in the simulation of Fig.~\ref{CED}. The largest possible value of $\sigma=\ln(1/\rho)\approx 0.110$ is attained when $N(s_k,s_k-\tau)=1$ for all $k\in\mathbb{N}$, that is, there are no impulses in each open interval $(s_k-\tau,s_k)$. Hence, the largest possible $T^*\approx 0.014$ could be obtained, which implies that an interval of length $0.014$ includes at least one impulse on average. Nevertheless, we have shown that the event times determined by the trigger~\eqref{ETM1} are lower bounded by $\Gamma\approx 0.063>T^*$, which means any two consecutive events (or impulses) are separated by at least 0.063 units of time. Hence, the impulses are triggered much less frequently by our event trigger~\eqref{ETM1}, when compared with the time-triggered impulsive control method in~\cite{KZ-EB:2021}. {For time-triggered impulsive control, the ADT conditions allow flexibility on the choice of impulse times. However, such conditions require that the impulse times should be pre-scheduled to ensure the average dwell time, and not all the impulses are necessary in order to preserve the desired performance of the control system. The proposed event-triggered impulsive control method only activates the impulses when it is needed to prevent the system from violating the desired performance. Hence, fewer impulses
are triggered in this example.}

\section{Conclusions}\label{Sec6}
This study investigated the event-triggered impulsive control problem of nonlinear systems. Actuation delays were considered with impulsive controllers. We proposed two types of event triggers depending on continuous and periodic event detection, respectively. Upper bounds of the actuation delays (and the sampling period) and sufficient conditions on the impulsive control inputs were derived to ensure asymptotic stability of the impulsive control systems. {A possible direction to improve the main results is to extend Assumption~\ref{assumption2} on linear rates $\rho$ and $\mu$ for the Lyaunov function to nonlinear rates, so that the proposed method can be applied to a wider class of nonlinear impulsive systems}, while a promising research direction is to generalize the designed impulsive control methods to stabilize time-delay systems. {Excluding Zeno behavior is one of the main objectives in the design of event-triggered control algorithms for the sake of practical implementations. Another possible research direction is to extend the solution beyond Zeno time for time-delay hybrid systems by following the line of work in \cite{SD-PF:2018}.}


\ifCLASSOPTIONcaptionsoff
  \newpage
\fi

%
%
%
%
%




\end{document}